\documentclass[12pt]{article}
\usepackage{amsmath, amsthm}
\usepackage{amscd}
\usepackage{amssymb}
\usepackage{array}

\usepackage{color}

\newcommand{\restr}[2]{{{#1}}_{{\left. \right|}_{#2}}}
\newcommand{\topored}[1]{\left| #1 \right|}
\newcommand{\p}[1]{|#1|}
\newcommand{\red}[1]{\tilde #1 }
\newcommand{\ev}{\mathrm{ev}}

\newtheorem{thm}{Theorem}[section]
\newtheorem{theorem}[thm]{Theorem}

\newtheorem{corollary}[thm]{Corollary}
\newtheorem{lemma}[thm]{Lemma}
\newtheorem{proposition}[thm]{Proposition}

\theoremstyle{definition}
\newtheorem{definition}[thm]{Definition}
\newtheorem{example}[thm]{Example}

\newtheorem{observation}[thm]{Observation}

\newtheorem{remark}[thm]{Remark}

\begin{document}

\newcommand{\id}{\relax{\rm 1\kern-.28em 1}}
\newcommand{\R}{\mathbb{R}}
\newcommand{\C}{\mathbb{C}}
\newcommand{\Z}{\mathbb{Z}}
\newcommand{\Q}{\mathbb{Q}}
\newcommand{\g}{\mathfrak{G}}
\newcommand{\e}{\epsilon}

\newcommand{\bp}{\mathbf{p}}
\newcommand{\bmax}{\mathbf{m}}
\newcommand{\bT}{\mathbf{T}}
\newcommand{\bU}{\mathbf{U}}
\newcommand{\bP}{\mathbf{P}}
\newcommand{\bA}{\mathbf{A}}
\newcommand{\bm}{\mathbf{m}}
\newcommand{\bIP}{\mathbf{I_P}}

\newcommand{\cJ}{\mathcal{J}}
\newcommand{\cF}{\mathcal{F}}
\newcommand{\cP}{\mathcal{P}}
\newcommand{\ep}{\mathcal{E}}
\newcommand{\cU}{\mathcal{U}}
\newcommand{\cH}{\mathcal{H}}
\newcommand{\cK}{\mathcal{K}}
\newcommand{\cO}{\mathcal{O}}
\newcommand{\cl}{\ell}
\newcommand{\cFG}{\mathcal{F}_{\mathrm{G}}}
\newcommand{\cHG}{\mathcal{H}_{\mathrm{G}}} 
\newcommand{\cD}{{\mathcal D}}

\newcommand{\lft}{\left(} %(
\newcommand{\rgt}{\right)} %)
\newcommand{\wt}{\widetilde}
\newcommand{\lfq}{\left[} %[
\newcommand{\rgq}{\right]} %]
\newcommand{\lfg}{\left\{} %{
\newcommand{\rgg}{\right\}} %}
\newcommand{\idshf}{\mathrm{id}}

\newcommand{\rGL}{\mathrm{GL}}
\newcommand{\rSU}{\mathrm{SU}}
\newcommand{\rSL}{\mathrm{SL}}
\newcommand{\rSO}{\mathrm{SO}}
\newcommand{\rOSp}{\mathrm{OSp}}
\newcommand{\rsl}{\mathrm{sl}}
\newcommand{\rM}{\mathrm{M}}
\newcommand{\M}{\mathrm{M}}
\newcommand{\End}{\mathrm{End}}
\newcommand{\Hom}{\mathrm{Hom}}
\newcommand{\HOM}{\mathrm{Hom}}
\newcommand{\diag}{\mathrm{diag}}
\newcommand{\rspan}{\mathrm{span}}
\newcommand{\rank}{\mathrm{rank}}
\newcommand{\Gr}{\mathrm{Gr}}
\newcommand{\ber}{\mathrm{Ber}}

\newcommand{\fsl}{\mathfrak{sl}}
\newcommand{\fg}{\mathfrak{g}}
\newcommand{\fh}{\mathfrak{h}}
\newcommand{\ff}{\mathfrak{f}}
\newcommand{\fgl}{\mathfrak{gl}}
\newcommand{\fosp}{\mathfrak{osp}}
\newcommand{\fm}{\mathfrak{m}}

\newcommand{\str}{\mathrm{str}}
\newcommand{\Sym}{\mathrm{Sym}}
\newcommand{\tr}{\mathrm{tr}}
\newcommand{\defi}{\mathrm{def}}
\newcommand{\Ber}{\mathrm{Ber}}
\newcommand{\spec}{\mathrm{Spec}}
\newcommand{\sschemes}{\mathrm{(sschemes)}}
\newcommand{\sschemeaff}{\mathrm{ {( {sschemes}_{\mathrm{aff}} )} }}
\newcommand{\sets}{\mathrm{(sets)}}
\newcommand{\supgrp}{\mathrm{(sgrps)}}
\newcommand{\shcps}{\mathrm{(shcps)}}
\newcommand{\Top}{\mathrm{Top}}
\newcommand{\sarf}{ \mathrm{ {( {salg}_{rf} )} }}
\newcommand{\arf}{\mathrm{ {( {alg}_{rf} )} }}
\newcommand{\odd}{\mathrm{odd}}
\newcommand{\alg}{\mathrm{(alg)}}
\newcommand{\sa}{\mathrm{(salg)}}
\newcommand{\SA}{\mathrm{(salg)}}
\newcommand{\salg}{\mathrm{(salg)}}
\newcommand{\varaff}{ \mathrm{ {( {var}_{\mathrm{aff}} )} } }
\newcommand{\svaraff}{\mathrm{ {( {svar}_{\mathrm{aff}} )}  }}
\newcommand{\ad}{\mathrm{ad}}
\newcommand{\Ad}{\mathrm{Ad}}
\newcommand{\pol}{\mathrm{Pol}}
\newcommand{\Lie}{\mathrm{Lie}}
\newcommand{\Proj}{\mathrm{Proj}}
\newcommand{\uspec}{\underline{\mathrm{Spec}}}
\newcommand{\uHom}{\underline{\mathrm{Hom}}}
\newcommand{\uproj}{\mathrm{\underline{Proj}}}
\newcommand{\sym}{\cong}

\newcommand{\al}{\alpha}
\newcommand{\be}{\beta}
\newcommand{\lam}{\lambda}
\newcommand{\ga}{\gamma}
\newcommand{\de}{\delta}
\newcommand{\lra}{\longrightarrow}
\newcommand{\ra}{\rightarrow}
\newcommand{\ua}{\underline{a}}
\newcommand{\coloneqq}{:}
\newcommand{\VEC}{\mathrm{Vec}}
\newcommand{\op}{\mathrm{op}}
\newcommand{\MAPSTO}{\mapsto}
\newcommand{\cinfty}{C^\infty}
\newcommand{\hotimes}{\hat \otimes}

%\today

\medskip

    \centerline{\LARGE \bf  Super Distributions, Analytic and Algebraic}

\bigskip

\centerline{\LARGE \bf Super Harish-Chandra pairs }

\medskip

\centerline{  C. Carmeli$^\natural$, R. Fioresi$^\flat$}

\bigskip

\centerline{\it $^\natural$Dipartimento di Fisica, 
Universit\`a di Genova, and INFN, sezione di Genova}
\centerline{\it Via Dodecaneso 33, 16146 Genova, Italy}
\centerline{{\footnotesize e-mail: claudio.carmeli@gmail.com}}

\bigskip

\centerline{\it $^\flat$ Dipartimento di Matematica, Universit\`a di
Bologna }
 \centerline{\it Piazza di Porta S. Donato, 5.}
 \centerline{\it 40126 Bologna. Italy.}
\centerline{{\footnotesize e-mail: fioresi@dm.UniBo.it}}

\section{Introduction}

The purpose of this paper is to extend the theory of Super Harish-Chandra
pairs, originally developed by Koszul for Lie supergroups \cite{koszul},
to analytic and algebraic supergroups, in order to obtain information also
about their representations. Along the same lines, we also want
to define the distribution superalgebra
for algebraic and analytic supergroups and
study its relation with the universal
enveloping superalgebra in analogy with Kostant's treatment for the
differential category \cite{Kostant}.

\medskip

Our intention is to provide different but equivalent approaches
to the study of analytic and algebraic supergroups and their actions 
over fields of characteristic zero.

\medskip

We realize that in both cases the theory is very similar to 
the differential one, however given some crucial differences between
the smooth category and the analytic and algebraic one,
we believe the present work is justified, given the importance 
of this theory for practical purposes together
with the lack of an appropriate and complete available
reference, though we are aware that a good step towards a
complete clarification of
these issues, for the analytic setting only,
appears in the papers \cite{vi1}, \cite{vi2}.

\bigskip

This paper was put on the web on June 2011. Since then Masuoka has published a
more general result and in his paper \cite{ma}
has quoted our work. Since our methods
are somewhat different from Masuoka's we hope 
our work deserves a place in the
literature.

\bigskip

This paper is organized as follows.

\medskip

In Section \ref{distr} we describe the superalgebra of distributions
of an analytic or an algebraic supergroup, establishing its relation with the
universal enveloping superalgebra.

\medskip

In Section \ref{shcp} we establish the equivalence between the
category of analytic or algebraic super Harish-Chandra pairs and the category
of analytic or affine algebraic supergroups under suitable hypothesis
for the ground field.

\medskip

In Section \ref{actions} we provide an equivalent
approach to the study of the actions of supergroups, via
the super Harish-Chandra pairs (SHCP).

\medskip

For all the definitions and main results in supergeometry
expressed with our notation,
we refer the reader to \cite{fg1} ch. 2 or
\cite{ccf} ch. 1, 4, 10. In particular we
shall employ both the sheaf theoretic and the functor of
points approach to supergeometry. On this we invite the
reader to consult the classical references \cite{DM}, \cite{ma},
\cite{vsv}.

\bigskip

{\bf Acknoledgements.} We wish to thank prof. Varadarajan, for
suggesting the problem and  prof. Cassinelli and prof.
Gavarini for helpful discussions.

\section{The superalgebra of distributions} \label{distr}

We start giving the definition of distribution and
distribution superalgebra. 
Our treatment is general enough to accommodate
the two very different categories of supermanifolds
and superschemes.
For the classical definitions we send the reader
to \cite{ja} pg 95, \cite{dg} II \S 4, no. 6 and \cite{dieu}.
For the basic definitions of supergeometry we refer the reader
to \cite{ma}, \cite{vsv}, \cite{DM}, \cite{fg1}.        

\subsection{Distributions}

Let $k$ be the ground field. 

\medskip

Let $X=(|X|,\cO_X)$ be an analytic supermanifold
or an algebraic superscheme over 
the field $k$.\footnote{If $X$ is an analytic supermanifold, $k=\R$ or
$k=\C$ or even $k=\Q_p$, the $p$-adic numbers (see for example
\cite{se}). If $X$ is a superscheme, $k$
is a generic field.}

\medskip

Let $X(k)$ be the $k$-points of $X$, 
that is $X(k)=\Hom(k^{0|0},X)$ in the functor of points notation. 
For an analytic supermanifold $X$ we have that its $k$-points $X(k)$
are identified with the topological points $|X|$, while for $X$ a superscheme
the $k$-points, are in one to one correspondence with the rational points,
that is, the points $x \in |X|$ for which $\cO_{X,x}/m_{X,x} \cong k$,
$m_{X,x}$ being the maximal ideal in the stalk $\cO_{X,x}$.

\begin{definition} \label{defdistr}
A \textit{distribution supported at $x \in X(k)$ of order at most n} is a morphism
$\phi: \cO_{X,x} \lra k$, with $m_{X,x}^{n+1} \subset ker(\phi)$ for some
$n$. 
The set of all distributions at $x$ of order $n$ is denoted as $D_n(X,x)$,  
while $D(X,x)$ denotes all distributions supported at $x$.
Both $D_n(X,x)$ and $D(X,x)$
have a natural super vector space structure.

We also define:
$$
D(X)=\cup_{x \in X(k)} D(X,x).
$$
as the \textit{distributions of finite order} of $X$. Also $D(X)$ has a natural
super vector space structure.
\end{definition}

\begin{observation} \label{obsdistr}
\begin{enumerate}

\item Notice that most immediately:
$$
D_n(X,x) \cong (\cO_{X,x}/m_{X,x}^{n+1})^*
$$
since if $\phi \in D_n(X,x)$, $\phi(m_{X,x}^{n+1})=0$, hence it factors and
becomes an element in $(\cO_{X,x}/m_{X,x}^{n+1})^*$.
Notice furtherly that:
$$
D_0(X,x)=k, \qquad D_1(X, x)=k \oplus (m_{X,x} \, \big/ \, m_{X,x}^2)^*.
$$
Hence $D_1(X,x)^+:=(m_{X,x} \, \big/ \, m_{X,x}^2)^*$
becomes identified
with the tangent space to $X$ at the point $x$.

\item If $X$ is an affine algebraic superscheme, 
$\cO(X)$ the superalgebra
of the global sections of its structural sheaf, a distribution 
supported at $x$ of order $n$
can be equivalently seen as a morphism 
$\phi: \cO(X) \lra k$, with $m_{x}^n \subset ker(\phi)$, where
$m_x:=\{\phi \in \cO(X) \, | \, \phi(x)=0\}$ 
is the maximal ideal of all the functions vanishing at $x$,
where as usual in supergeometry $f(x)$ simply means the 
image in $\cO_{X,x}/m_{X,x}$ of the element $f \in \cO(X)$ 
under the natural morphisms: $\cO(X) \lra \cO_{X,x} \lra 
\cO_{X,x}/m_{X,x} \cong k$. (Notice that
since $x$ is rational, we have $\cO(X)=k \oplus m_x$ and
$\cO_{X,x}/m_{X,x} \cong k$). 

We leave to the reader to check that the two given definitions
of distributions are essentially the same in this case.

\medskip

\item If $X$ is a smooth supermanifold, that is, if we are in the differential
category, we can view a point supported distribution as
a morphism $\phi:\cO(X) \lra \R$, $J_{x}^n \subset ker(\phi)$,
where $J_{x}$ is the maximal ideal corresponding to the point $x \in |X|$
(see \cite{Kostant} and
\cite{ccf} 4.7), thus recovering the same definition as in $(2)$
for the affine algebraic category.
This is one of the many analogies between the category of affine
supervarieties and smooth supermanifolds.
\end{enumerate}

\end{observation}

\begin{example} {\sl The distributions on $k^{p|q}$}. \label{exdistraffine}
($char(k)=0$).
Consider the superspace $k^{p|q}$ 
(both in the analytic and affine algebraic context). 
Let $x_1 \dots x_p, \xi_1 \dots \xi_q$
denote the global coordinates and $J_0=(x_1 \dots x_p, \xi_1 \dots \xi_q)$
the maximal ideal in the stalk $\cO_{X,0}$ at the origin.
We have that 
$$
\cO_{X,0}/J_0^{n+1} \cong {\mathrm span_k} 
\{1, \, x^{i_1}_1 \dots x^{i_p}_p\xi_1^{i_{p+1}} \dots \xi_q^{i_{p+q}}, 
\, \sum i_k=n \, \}.
$$
Let $X^I$ denote the monomial
$x^{i_1}_1 \dots x^{i_p}_p\xi_1^{i_{p+1}} \dots \xi_q^{i_{p+q}}$,
$I=(i_1 \dots i_{p+q})$.
Since the distributions at $0$ of order $n$
are the dual of the super vector space
$\cO_{X,0}/J_0^{n+1}$, we have that a basis for the super vector space of
distributions at the point $0$ is given by
$\phi_J$ %with $J=(j_1 \dots j_{p+q})$, 
such that $\phi_{J}(X_I)
=\delta_{IJ}$, with $I=(i_1 \dots i_{p+q})$, $J=(j_1 \dots j_{p+q})$
multiindices, $\sum i_k=\sum j_k=n$.
So we have:
$$
\phi_{j_1 \dots j_{p+q}}(f)={1 \over j_1! \dots j_{p+q}!}
\big( {\partial \over \partial x_1} \big)^{j_1}  \dots
\big( {\partial \over \partial x_p} \big)^{j_p}  
\big( {\partial \over \partial \xi_q} \big)^{j_{p+1}} \dots  
\big( {\partial \over \partial \xi_{1}} \big)^{j_{p+q}}(f)(0).
$$

\end{example}

\subsection{The superalgebra of distributions of
an analytic supermanifold}

In this section we want to characterize the distributions
for an analytic supermanifold $M=(|M|, \cO_M)$ in the following way.
Distributions at the point $x\in |M|$ are the elements 
in $\cO_{M,x}^*$, whose kernel contains an ideal of finite 
codimension
in analogy with Kostant's treatment for
the smooth category (see \cite{Kostant}). 
Let us see this more in detail, we start with a lemma.

\begin{lemma}
Let $M=(|M|, \cO_M)$ be an analytic supermanifold, $x \in |M|$,
$I_x$ the ideal in $\cO_{M,x}$ of the sections vanishing at $x$.
For each positive integer $p$, $I_x^p$ is an ideal of finite codimension.
\end{lemma}
\begin{proof}
It follows from the Taylor expansion formula. 
In fact, every element $f$ in $\cO_{M,x}$ can be written as
$f= \sum_I f_I \theta^I$, 
where $f_I$ is an element in the classical stalk of germs of 
holomorphic functions
${\cal H}_{M,x}$. For each positive integer $q$, a germ $f_I$ 
can in turn  be written as 
\[
f_I(z)=f_I(x)+\sum_{K, |K|=1}^{q-1}( \partial_K f_I)(x) z^K+
\sum_{J , |J|=q} z^Jh_{I,J}(z)
\]
where $I$, $J$, $K$ are multiindeces.
Hence we can write 
\[
f= \sum_{I} \left(f_I(x)+\sum_{R, |R+I|<p} (\partial_R f_I)(x)z^R
\right)\theta^I+ 
\sum_{|I+R|=p}h_{I,R}(z)z^R\theta^I.
\]
From this formula, it follows that the elements in $I_x^p$ are generated 
by the monomials $\{z^K \theta^I\}_{|K+I|\leq p}$.
\end{proof}

\begin{proposition}
An ideal $J$ in $\cO_{M,x}$ has finite codimension if and only if there 
exists an integer $p>0$ such that
$I_x^p\subseteq J$.
\end{proposition}

\begin{proof}
The ``if'' part follows from the previous lemma.
For the ``only if'' part we reason as follows. Consider the descending 
chain of ideals
$J+I^p_x\supseteq J+I^{p+1}_x$. Since $J$ has finite codimension there 
exists $q$ such that
$J+I^q_x = J+I^{q+1}_x$. From this it follows that 
$I^q_x \subseteq J+ I^q_x\cdot I_x$.  
Since, by the previous lemma, $I^q_x$ is finitely generated 
we can apply the super version of Nakayama lemma 
(see \cite{vsv}) and we get $I^q_x\subseteq J$.
\end{proof}

We have then obtained the following result, which establishes a 
parallelism with the smooth category.

\begin{theorem}
The distributions on an analytic supermanifold $M$ supported at
a point $x$ correspond to the morphisms $f: \cO_{M,x}\lra k$, whose
kernel contains an ideal of finite codimension.
\end{theorem}

\subsection{The distributions of a supergroup at the identity}

We now want to restrict our attention to the distributions 
of a supergroup (analytic or algebraic) at the identity element $e \in G(k)$.

As a consequence of the observation \ref{obsdistr} 
we have that:
$$
D_1(G,e)^+ \cong (m_{G,e}/m_{G,e}^2)^* \cong T_e(G)=\Lie(G).
$$
It is only natural to expect
$D(G,e)$ to be identified with $\cU(\fg)$, with $\fg=\Lie(G)$. 
This is true, as we
shall see, provided we exert some care.

\medskip

As we remarked in the definition \ref{defdistr}
the distributions at the identity 
are a super vector space, however 
there is a natural additional
superalgebra structure that we can associate to
the super vector space of distributions, by defining the
{\sl convolution product}.

\begin{definition} \label{defconvprod}
Let $\phi$, $\psi \in D(G,e)$. We define their
\textit{convolution product} as the following morphism:
$$
(\phi \star \psi)(f)=(\phi \otimes \psi)\mu^*(f), \qquad f \in \cO_{G,e}
$$
where $\mu$ denotes the multiplication in the supergroup $G$
and $\mu^*$ the corresponding sheaf morphism.
\end{definition}

The following proposition is a straightforward check.

\begin{proposition} \label{propconvprod}
The convolution product makes $D(G,e)$ a superalgebra, its
unit being the evaluation at $e$, $ev_e:\cO_{G,e} \lra k$.
\end{proposition}

\medskip

We now want to examine the relation of $D(G,e)$ with the universal
enveloping superalgebra (uesa) of the supergroup $G$.
Since $D(G,e)$ $\supset D_1(G,e)^+$ $\cong$ $\Lie(G)$,
by the universal property of the uesa $\cU(\fg)$ we have
a superalgebra morphism $\al:\cU(\fg) \lra D(G,e)$.

\begin{observation} 
When $G$ is an algebraic supergroup and
the characteristic of $k$ is positive, as it happens
in the classical setting, $D(G,e)$ contains more than the elements
coming from $\cU(\fg)$ (refer to example \ref{exdistraffine}).
This is because the divided powers $X^m/m!$
are in $D(G,e)$ but not in 
$\cU(\fg)$. Again similarly, as in the classical situation, we have
that any morphism $\cU(\fg) \lra D(G,e)$ factors via \textit{ the
universal enveloping restricted algebra} $\cU^r(\fg)$:
$$
\cU(\fg) \lra \cU^r(\fg)=\cU(\fg)/(X^p-X^{[p]}) \lra D(G,e).
$$
where $X^{[p]}$ denotes the derivation in $\fg$ corresponding to
$p$-times the derivation $X$ (which is a derivation here, since we
are in characteristic $p$).
\end{observation}

Let $char(k)=0$.

\medskip

\begin{proposition} \label{isoDU}
The morphism $\al:\cU(\fg) \lra D(G,e)$ is an isomorphism.
\end{proposition}

\begin{proof}
This is done essentially in the same way as in the classical setting, 
which is detailed in \cite{vsv} ch. I for the
analytic category and \cite{dg} II, 6, 1.1
for the algebraic category. 
\end{proof}

\begin{proposition} \label{leftinv-operators}
There is an isomorphism of the superalgebra of
distributions on a supergroup $G$ 
and the superalgebra of the left invariant differential
operators on $G$.
$\cU(\fg)$ is then
isomorphic to the superalgebra of the left invariant differential
operators on $G$.
\end{proposition}

\begin{proof} Again this proof is the same
as in  \cite{vsv} ch. I and 
\cite{dg} II, 6, 1.1, for the classical setting. 
\end{proof}

\subsection{The distributions of
an affine algebraic supergroup}

We now want to 
restrict ourselves to the case of affine
algebraic supergroups.
As we shall see, this algebraic setting shares many similarities
with the differential one.

\medskip

Consider the module of distributions $D(G)$ (see 
observation \ref{obsdistr}):
$$
D(G)=\cup_{x \in G(k)} D(G,x) \subset \cO(G)^*.
$$

\begin{definition}
If $\phi=\sum \phi_{p_i}$ is a distribution with $\phi_{p_i} \in D(G,p_i)$
we say that $\phi$ is \textit{supported} at $\{p_i\}$.
On the whole $D(G)$ we have a well defined 
associative product, called the \textit{convolution product}:
$$
(\phi_p \star \phi_q)(f)=(\phi_p \otimes \phi_q)\mu^*(f)
$$
and its unit is $ev_e$, the evaluation at the unit element:
$ev_e(f)=f(e)$. $\mu^*$ denotes (as before) the comultiplication
in the Hopf superalgebra $\cO(G)$.
\end{definition} 

\begin{observation}
If $\phi_p$ and $\phi_q$ are two distributions supported at $p$ and $q$
respectively, then $\phi_p \star \phi_q$ is supported at $pq$. This
is a consequence of the fact: 
$$
\mu^*(m_{pq}) \subset m_p \otimes \cO(G)+\cO(G) \otimes m_q
$$
where $m_x$ is as usual the maximal ideal of the sections in $\cO(G)$ vanishing
at $x \in G(k)$. 
$m_x=m_{x,0}+ J_{O(G)}$, that is, $m_x$ is the sum of $m_{x,0}$ the ordinary
maximal ideal corresponding to the topological 
rational point $x \in G(k)$ and
the ideal $J_{\cO(G)}$ generated by the odd sections in $\cO(G)$. 
\end{observation}

\medskip

\begin{lemma}
Let $\phi_g \in D(G,g)$. Then there exists 
a unique $\phi_e \in D(G,e)$
such that $\phi_e=ev_{g^{-1}} \star \phi_g$.
\end{lemma}

\begin{proof}
Since 
$\phi=(ev_g \star ev_{g^{-1}}) \star \phi$, define 
$\phi_e=ev_{g^{-1}} \star \phi \in D(G,e)$.
\end{proof}

\begin{proposition} \label{distr-hopf}
$D(G)$ is super Hopf algebra with comultiplication $\Delta$, counit $\eta$ and
antipode $S$ given
by:
$$
\Delta(\phi_g)(f \otimes g):=\phi_g(f\cdot g) \qquad
\eta(\phi_g)(f):=\phi_g({\rm ev}_e(f))
$$
\begin{align*}
S(\phi_g)(f) & := \phi_g(i^\ast (f)),
\end{align*}
where $i:G \lra G$ denotes the inverse morphism.
\end{proposition}

\begin{proof} Direct check. \end{proof}

Let $k|G|$ be the group algebra corresponding to the ordinary
group $G(k)$. In other words 
$$
k|G|=\left\{ \sum_{g \in G(k), \lambda_g \in k} \lambda_g g
\right\}. 
$$

\begin{proposition} \label{iso}
We have a linear isomorphism:
$$
\begin{array}{cccc}
\Psi: & D(G) & \lra & k|G| \otimes \cU(g) \\
& \phi_g & \mapsto & g \otimes \phi_e
\end{array}
$$
that endows $k|G| \otimes \cU(\fg)$ of a Hopf superalgebra structure.
This structure is induced by the natural Hopf structures on the
group algebra $k|G|$ and $\cU(\fg)$:
$$
\Delta_{k|G|}(g)=g \otimes g, \qquad
\Delta_{\cU(\fg)}(U)=U \otimes 1+1 \otimes U, \qquad g \in G(k), U \in \fg.
$$
The superalgebra structure is defined as:
$$
(g \otimes X)(h \otimes Y)=gh \otimes (h^{-1}X)Y, \quad
g \in G(k), \qquad X,Y \in \cU(\fg)
$$
with $h^{-1}X:=ev_{h^{-1}} \star X \star ev_h$, 
(by \ref{isoDU} we identify distributions at $e$ with elements
in $\cU(\fg)$).
\end{proposition}

\begin{proof} This is done with a direct check. We just point out
that it is enough to do such check just on generators.
\end{proof}

\section{Super Harish-Chandra Pairs} \label{shcp}

The theory of Super Harish-Chandra Pairs (SHCP) that we shall 
develop presently provides an equivalent way to approach
the analytic or affine algebraic supergroups.

\medskip

\subsection{Definition of SHCP}

Any time we say {\sl supergroup} we mean an analytic or 
an affine algebraic supergroup over a field $k$ of characteristic zero.

\begin{definition} \label{def:SHCP}
Suppose $(G_0,\fg)$ are respectively a group (analytic or
affine algebraic) and a super
Lie algebra. Assume that:
\begin{enumerate}
\item $\fg_0 \simeq {\rm Lie}(G_0)$,
\item  $G_0$ acts on $\fg$ and this action restricted to
$\fg_0$ is the adjoint representation of $G_0$ on
    $\Lie(G_0)$. Morever the differential of such action is
the Lie bracket. We shall denote such an action with $\Ad$ or
as $g.X$, $g \in G_0$, $X \in \fg$. 
\end{enumerate}
Then $(G_0,\fg)$ is called a \emph{super Harish-Chandra pair
(SHCP)}.

\medskip

A \textit{morphism} of SHCP is simply a pair of morphisms 
$\psi=( \psi_0,\rho^\psi )$ preserving the SHCP structure
that is:

\begin{enumerate}
\item $\psi_0 : G_0 \rightarrow H_0$ is a group morphism (in the
analytic or algebraic category);
\item $\rho^\psi:\fg \rightarrow \fh$ is a super Lie algebra morphism
\item $\psi_0$ and $\rho^\psi$ are compatible in the sense that:
\begin{eqnarray*}
\rho^\psi_{\left.\right| \fg_0} & = & 
d\psi_0\qquad \Ad(\psi_0(g))\circ\rho^\psi= \rho^\psi\circ\Ad(g)
\end{eqnarray*}
\end{enumerate}

When $G_0$ is an analytic group we shall speak of an 
\textit{analytic SHCP}, when $G_0$ is an affine algebraic group
of an \textit{algebraic SHCP}.

\end{definition}

We would like to show that the 
category of (analytic of algebraic) SHCP (denoted with $\shcps$)
is equivalent to the category of 
supergroups (analytic or algebraic), denoted with $\supgrp$. 
In order to do this we start by associating in a natural
way a supergroup to a SHCP.

\begin{definition} \label{shcpsheaf}
Let $(G_0,\fg)$ be a SHCP. The sheaf $\cO_{G_0}$
of the ordinary group $G_0$ carries  
a natural action of $\cU(\fg_0)$, since the elements
of $\cU(\fg_0)$ act on the sections in $\cO_{G_0}(U)$
as left invariant differential operators. 
We define $\cO_G(U)$ as:
$$
\cO_G(U):=\Hom_{\cU(\fg_0)}(\cU(\fg), \cO_{G_0}(U)), \qquad
U \subset_{open} G_0.
$$
\end{definition}

\begin{proposition}
The assignment $U \mapsto \cO_G(U)$ is a sheaf of superalgebras
on $G_0$, where the superalgebra structure on $\cO_G(U)$ is
given by:
$$
f_1 \cdot f_2=m_{\cO_{G_0}} \circ (f_1 \otimes f_2) \circ \Delta_{\cU(\fg)} 
$$
and the restriction morphisms $\rho_{UV}:\cO_G(U) \lra \cO_G(V)$ are
$\rho_{UV}(f):= {\widetilde \rho_{UV}} \circ f$ where
$ {\widetilde \rho_{UV}} $ are the restrictions of the ordinary sheaf
$\cO_{G_0}$.
\end{proposition}

\begin{proof}
The check $f_1 \cdot f_2$ is an associative product is routine, while
the sheaf property comes from the fact $\cO_{G_0}$ is an ordinary sheaf.
%(for more details see ch. 7 in \cite{ccf}).
\end{proof}

We now show that $(G_0, \cO_G)$ is a superspace, by showing that
is {\sl globally split}, in other words that:
$$
\cO_G(U) \cong \cO_{G_0}(U) \otimes \wedge(\fg_1).
$$

%This will ensure that $(G_0, \cO_G)$ is a supergroup, analytic
%or algebraic, depending whether $G_0$ is analytic or algebraic.

\begin{theorem} \label{gamma}
\begin{enumerate}
    \item The map
    \[
        \begin{aligned}
            \widehat{\gamma} \colon \cU(\fg_0) \otimes \wedge(\fg_1) &
            \to \cU(\fg) \\
            X \otimes Y &\mapsto X \cdot \gamma(Y)
        \end{aligned}
    \]
    is an isomorphism of super left $\cU(\fg_0)$-modules, where
    \[
\begin{aligned}
            \gamma \colon \wedge(\fg_1) &\to \cU(\fg) \\
            X_1 \wedge \cdots \wedge X_p &\mapsto
            \frac{1}{p!} \sum_{\tau\in{{S}_p}} (-1)^{|\tau|} X_{\tau(1)} 
\cdots X_{\tau(p)}
        \end{aligned}
    \]
is the symmetrizer map, $|\tau|$ denotes the parity of the
permutation $\tau$.

\item
$(G_0, \cO_G)$ is globally split
i.~e. for each open subset $U\subseteq G_0$, there is an 
isomorphism of superalgebras
\begin{equation} \label{eq:global_split}
    \cO_G(U)
    \simeq \Hom \big(\wedge(\fg_1), \cO_{G_0}(U)\big)
    \simeq \cO_{G_0}(U) \otimes {\wedge(\fg_1)}^\ast.
\end{equation}
Hence $\cO_G$  carries a natural $\Z$-gradation.
\end{enumerate}
\end{theorem}

\begin{proof} (1) is an application of Poincar\'e-Birkhoff-Witt
(PBW) theorem (see \cite{vsv}), while for (2) 
consider the following map:
\begin{align*}
\phi_U:\, \cO_G(U) & \to \Hom \big(\wedge(\fg_1),\cO_{G_0}(U)\big)\\
f & \to f\circ \gamma
\end{align*}
 Since $\gamma$ is a supercoalgebra morphism,
 $\phi_U$ is a superalgebra morphism. In fact:
$$
\phi_U(f_1 \cdot f_2)=m \circ f_1 \otimes f_2 \circ
\Delta_{\cU(\fg)} \circ \gamma=
m \circ f_1 \otimes f_2 \circ (\gamma \otimes \gamma)
\Delta_{\cU(\fg)}=\phi_U(f_1)\phi_U(f_2).
$$
The fact that $\phi_U$   is a  superalgebra 
isomorphism follows at once from $\cU(\fg_0)$-linearity.
\end{proof}

As an almost immediate consequence of the previous theorem
we have the following corollary.%, which is crucial in our treatment.

%This last result shows that 
\begin{corollary}
If $G_0$ is an analytic
manifold (resp. algebraic scheme), then $(G_0, \cO_{G})$
is a superspace. 
\end{corollary}

In the next sections we will complete the task of showing 
$(G_0, \cO_{G})$ is a supergroup, by providing explicit expression
for the multiplication, unit and inverse. This will lead to the main
result of the paper, namely the equivalence of categories between
the SHCP and supergroups. We now state the main result of the paper and then
we shall prove it with different methods in the next sections, since
at this points the analytic and algebraic categories diverge and
require a dramatically different treatment.

\begin{theorem} \label{eqcat}
Let $k$ be a field of characteristic zero, algebraically closed
if we are in the algebraic category.
Define the functors
$$
\begin{array}{ccc}
\cH:\, \supgrp & \to & \shcps \\
G & \to &({G_0}, \Lie(G) )\\
\phi & \to & (|\phi|, (d \phi)_e)
\\ \\
\cK:\, \shcps & \to & \supgrp \\
(G_0,\fg) & 
\to & {\bar G}:=(G_0, \HOM_{\cU(\fg_0)} \big( \cU(\fg), \cO_{G_0} \big)) \\
\psi=(\psi_0,\rho^\psi) & \to & f \mapsto \psi^*_0 \circ f \circ \rho_\psi
%\end{array}
\end{array}
$$
where $G$ and $(G_0, \fg)$ are objects and
$\phi$, $\psi$ are morphisms of the corresponding categories
(in the definition of $\cH$, $G_0$ is the ordinary group underlying $G$).
Then $\cH$ and $\cK$ define an equivalence between the categories
of supergroups (analytic or algebraic) and super Harish-Chandra pairs
(analytic or algebraic).
\end{theorem}

\subsection{Analytic SHCP} \label{analytic-shcp}

Let $k=\R$ or $\C$.

\medskip

For analytic SHCP it is relatively easy to define a supergroup structure
on the superspace $(G_0, \cO_G)$ we have defined above, by
mimicking what happens in the smooth case.
In fact for an analytic ordinary group $G_0$, the action
of $\cU(\fg_0)$ on $\cO_{G_0}$ %in \ref{shcpsheaf} 
is given by:
$$
(\wt{D}^{}_Z \cdot f)(g)=f(ge^{tZ}), 
\qquad Z \in \fg_0, \quad f \in \cO_{G_0}(U)
$$
where $e^{tZ}$ denotes the one-parameter subgroup corresponding
 to the element $Z \in \fg_0$.
%Notice that at this point we encounter an important difference with
%the algebraic setting, since in that case we do not have
%available such result, namely the Frobenius theorem.

\begin{proposition} \label{analsgrps}
$(G_0,\cO_G)$ is an analytic supergroup 
where the multiplication $\mu$, inverse $i$ and unit $e$
and are defined via the corresponding sheaf morphisms 
as follows. 
\begin{align}
\big[ \mu^\ast (f) (X,Y) \big] (g,h) &= 
\big[ f \big( (h^{-1}.X) Y \big) \big](gh) \\
    \label{eq:inv}
\big[ i^\ast (f) (X) \big] (g^{-1}) &=  
\big[ f(g^{-1}.\overline{X}) \big](g) \\
e^\ast(f) &=  \big[ f(1) \big](e)
\end{align}
for $f \in \cO_G(U)$, $g,h \in |G|$, where $|G|$
is the topological space underlying $G_0$. $\overline{X}$ denotes
the antipode in $\cU(\fg)$.
\end{proposition}

{\bf Note.} We shall discuss the peculiar form of $\mu^*$,
$i^*$, $e^*$ in remark \ref{duality}.
 
\begin{proof} 
The proof of this result is the same as in the differential smooth
setting, where everything is defined in the same way (see \cite{ccf} ch. 7). 
In particular
to prove that $\mu^\ast$, $i^\ast$, $e^\ast$ are $\cU(\fg_0)$-morphisms
is harder than the verification of the compatibility conditions
and the Hopf superalgebra properties.  
As an example, let us verify $\mu$ is well defined
the other
properties being essentially the same type of calculation.
Due to PBW theorem, it is  enough to prove $\fg_0$-linearity. Let hence $Z\in \fg_0$
\begin{eqnarray*}
\mu^\ast \lft f\rgt\lft ZX, Y\rgt\lft g, h\rgt& = &  f \lft h^{-1}\lft ZX\rgt Y\rgt(gh)\\
&= &f\lft (h^{-1}.Z )(h^{-1}.X) Y\rgt(gh)\\
& =& \wt{D}^{}_{h^{-1}.Z}\lfq f\lft (h^{-1}.X) Y\rgt\rgq (gh)
\end{eqnarray*}
on other hand
\begin{eqnarray*} 
\lfq (\wt{D}^{}_{Z}\otimes \idshf)\lft \mu^\ast\lft f\rgt\lft X,Y\rgt\rgt\rgq(g,h) &=& \frac{d}{d t}_{\left.\right|t=0} f\lft (h^{-1}X)Y  \rgt\lft ge^{tZ}h\rgt\\
& = &  \frac{d}{d t}_{\left.\right|t=0} f\lft (h^{-1}X)Y  \rgt\lft gh e^{t (h^{-1}Z)}\rgt\\
& = & \wt{D}^{}_{h^{-1}Z}\lfq f\lft (h^{-1}.X) Y\rgt\rgq (gh). 
\end{eqnarray*}
Similarly for the left entry, one finds:
\begin{eqnarray*}
\mu^\ast\lft f\rgt \lft X, ZY\rgt (g,h) & = & f\lft (h^{-1}X) ZY\rgt(gh)\\
& = & f\lft Z (h^{-1}X) Y +\lfq h^{-1}X, Z\rgq Y\rgt\lft gh\rgt \\
& = & \wt{D}^{}_Z \lft f\lft (h^{-1}X) Y\rgt\rgt (gh) + \\ 
& & f\lft \lfq h^{-1}X, Z\rgq Y\rgt\lft gh\rgt 
\end{eqnarray*}
and
\begin{eqnarray*}
\frac{d}{dt}_{\left.\right|t=0}\mu^\ast\lft f\rgt \lft X,Y\rgt \lft g, he^{tZ}\rgt & = & 
\frac{d}{dt}_{\left.\right|t=0} f \lft ((he^{tZ})^{-1} X) Y\rgt \lft g he^{tZ}\rgt\\
& = &\lfq \wt{D}^{}_Zf\lft (h^{-1}X)Y\rgt\rgq(gh)+ \\
& + & f\lft \lfq (h^{-1}X), Z\rgq Y \rgt (gh).
\end{eqnarray*}
\end{proof}

We are now ready for the proof of theorem \ref{eqcat}
in the analytic setting.

\begin{theorem} \label{anal-eqcat}
There is an equivalence of categories between analytic SHCP 
and analytic supergroups expressed by the functors
$\cK$ and $\cH$ in \ref{eqcat}.
\end{theorem}

\begin{proof}
Let us first show the correspondence between morphisms. 
If $\phi$ is a morphisms of analytic supergroups, it is immediate
that $(|\phi| ,(d\phi)_e)$ is a morphism of SHCP.
Vice versa, if $\psi=(\psi_0, \rho_\psi)$ is a morphism of SHCP
$(G_0, \fg)$, $(H_0, \fh)$, then $\psi^*:\cO_H(U) \lra \cO_G(\psi_0^{-1}(U))$
defined as $\psi^*(f)=\psi^*_0 \circ f \circ \rho_\psi$ is a sheaf morphism
and $(\psi_0, \psi^*)$ is a morphism of the supergroups $G$ and $H$.
As one can check the assignments detailed in \ref{eqcat} establish a one-to-one
correspondence between the sets of morphisms of SHCP and analytic
supergroups.

\medskip

We now turn to the correspondence between the objects.
Let $G$ be a supergroup and $\overline{G}$ the supergroup obtained
from the SHCP $(G_0, \Lie(G))$, where $G_0$ is the ordinary analytic
group underlying $G$.
As for the smooth setting, let us define the morphism
$\eta \colon \overline{G} \to G$ %and prove that it is a SLG
%isomorphism.
\[
    \begin{aligned}
        \eta^* \colon \cO_G(U) &\to \cO_{\overline{G}}(U) = 
\HOM_{\cU(\fg_0)} \big( \cU(\fg), \cO_{G_0}(U) \big) \\
        s &\mapsto \Big( \overline{s} \colon X \to (-1)^{|X|}|{(D^{}_X s)}| \Big).
    \end{aligned}
\]
Here $D^{}_X$ denotes the left invariant differential operator on $G$ 
associated with $X\in \cU(\fg)$, that is $D^{}_X=(1 \otimes X)\mu^*$.
The definition is well posed as one can directly check, moreover
$\eta$ is a SLG morphism, i.~e.
\[
    \eta \circ \mu_{\overline{G}} = \mu_G \circ (\eta \times \eta)
\]
Indeed, for each $s \in \cO(G)$, $X,Y \in \cU(\fg)$, and $g,h \in
G_0$,
$$
\begin{array}{rl}
    \big[ \big( (\eta^* \otimes \eta^*)\mu_G^*(s) \big)(X,Y) \big](g,h)
    &= (-1)^{|{X}|+|{Y}|} |(D^{}_X \otimes D^{}_Y) \mu_G^*(s)|(g,h) \\ \\
    &= (-1)^{|{X}|+|{Y}|} |{D^{}_{h^{-1}.X} D^{}_Y s }|(gh) \\ \\
    &= \big[ \eta^*(s)\big( (h^{-1}.X)Y \big) \big](gh) \\ \\
    &= \big[ \big( \mu_{\overline{G}}^* \eta^*(s) \big) (X,Y) \big](g,h).
\end{array}
$$
Now the last thing to check is that $\eta$ is an isomorphism. 
This is due to the fact that $|\eta|$ is clearly
bijective and, for each $g \in G_0$, the differential
$(d\eta)_g$ is bijective as
$$
\begin{array}{rl}
    \big[ (d \eta)_g ({{\overline {D}}^{}_X}_{g}) \big] (s)
&    = {\overline{D}^{}_X}_{g} \eta^*(s) =
ev_g({\overline D}^{}_X\eta^*(s))=[{\overline D}^{}_X\eta^*(s)](1)(g)= \\ \\
&    
    = (-1)^{|{X}|} \eta^*(s)(X)(g)
    = |{(D^{}_X s)}|(g)
    = {D^{}_X}_{g}(s)
\end{array}
$$
where we denote with ${\overline D}^{}_X$ a left invariant differential
operator on $\overline{G}$ corresponding to $X \in \cU(\fg)$
while $D^{}_X$ denotes a left invariant differential operator on $G$. 

\medskip

We conclude using the inverse function theorem, which holds also
for analytic supermanifolds and again this is an important difference
with the algebraic setting, where we do not have this tool available.
\end{proof}

\vfil\eject

\begin{remark} {\sl $p$-adic SHCP}.

$p$-adic supermanifolds, supergroups
and SHCP can be defined through the obvious same definitions 
within the framework described classically by
Serre in \cite{se}. In fact since the category of $p$-adic manifolds
resembles very closely the category of analytic manifolds, it is then
only reasonable to expect that one can develop along the same lines
the theory of $p$-adic supermanifolds. Once the basic results,
like the inverse function theorem, are established, the equivalence
of categories between $p$-adic supergroups and the $p$-adic
SHCP will then follow through the same proof we have
detailed for the analytic category.
\end{remark}

\subsection{Algebraic SHCP} \label{algebraic-shcp}

We now prove our main result, namely 
the theorem \ref{eqcat} in the case of $G$ an affine
algebraic supergroup over a field of characteristic zero, algebraically
closed. \\
The category of affine algebraic supergroups is 
equivalent to the category of commutative Hopf superalgebras, 
hence we need to show that
there is a 
unique commutative Hopf superalgebra $\cO(G)$ associated to a
SHCP $(G_0,\fg)$, namely the superalgebra of the global sections
of the sheaf $\cO_G$ as it is defined in \ref{def:SHCP}.

\medskip

We would like to state and prove the algebraic analogue of 
proposition \ref{analsgrps}. In the proof of such proposition
there is an essential use of the exponential, hence we now want to
formally introduce this notation in the algebraic setting,
so that we can reproduce all the arguments, though being well
aware that the exponential notation has a very different interpretation
in the two categories analytic and algebraic.

\medskip

Notice that since the exponential appears for the action of $\cU(\fg_0)$
on $\cO(G_0)$ (see beginning of sec. 
\ref{analytic-shcp}), the question is entirely classical and it
is treated in detail in \cite{dg} ch. 2 for the algebraic
setting. We shall briefly review
few key facts, sending the reader to \cite{dg} for all the details.

\medskip

Let $G_0$ be an algebraic group and $A$ a commutative
algebra, $p:A(t) \lra A$, $t^2=0$ the natural projection, 
$t$ even. By definition
$\Lie(G_0)(A)$ $=$ $\ker G_0(p)$. Since $G_0$ is affine we
have $G_0 \subset \rGL(V)$ for a suitable vector space $V$, hence
we can write:
$$
\begin{array}{rl}
\Lie(G_0)(A) &= \left\{\, 1+ tZ \, \right\} \subset G_0(A(t))
\subset \rGL(V)(A(t)) \\ \\ 
&= \rGL(V)(A)+t\End(V)(A)
\end{array}
$$
for suitable $Z \in \End(V)(A)$, where $\End(V)$ is the functor of points
of the superscheme of the endomorphisms of the vector space $V$. Very often
$\Lie(G_0)$ is identified with the subspace in $\End(V)$ consisting
of the elements $Z$.
As a notation device we define: 
$$
e^{tZ}=1+tZ \in G_0(A(t)).
$$
%As another common notational device, 
Let $g \in G_0(A)=\Hom(\cO(G_0),A)$, that is,
$g$ is an $A$-point of $G_0$, and let $f \in \cO(G_0)$. 
As another common notational device, we denote
$g(f)$ with $f(g)$. Since $A$ embeds naturally in $A(t)$ 
we can view $g$ also as an $A(t)$-point of $G_0$ and consider
$f(ge^{tZ})$. We then define:
$$
\frac{d}{dt}_{\left.\right|t=0} f(ge^{tZ})=b, \quad
\hbox{where} \quad f(ge^{tZ})=(ge^{tZ})(f)=a+bt \in A(t).
$$
With such definition one sees that 
$\frac{d}{dt}_{\left.\right|t=0} f(ge^{tZ})$ corresponds to the natural
action of $Z \in \Lie(G_0)$ on $\cO(G_0)$ via left invariant operators, that is
$$
\frac{d}{dt}_{\left.\right|t=0} f(ge^{tZ})=(1 \otimes Z)\mu^*(f)
$$
that we denoted with $\wt{D}^{}_Zf$ in the analytic category.

\medskip

We now go back to the super setting and prove the analogue of
proposition \ref{analsgrps}.

\begin{proposition} \label{algHopfstructure}
The superalgebra $\cO(G)=\Hom(\cU(\fg),\cO(G_0))$ 
associated to the algebraic SHCP $(G_0,\fg)$
is an Hopf superalgebra 
where the comultiplication $\mu^*$, antipode $i^*$ and counit $e^*$
\footnote{In analogy with proposition \ref{analsgrps} we have kept
the terminology $\mu^*$, $i^*$, $e^*$, though we are not making (yet)
any claim on the sheaf morphisms.}
and are defined as follows: 
\begin{align}
\big[ \mu^\ast (f) (X,Y) \big] (g,h) &= 
\big[ f \big( (h^{-1}.X) Y \big) \big](gh) \\
    \label{eq:inv}
\big[ i^\ast (f) (X) \big] (g^{-1}) &=  
\big[ f(g^{-1}.\overline{X}) \big](g) \\
e^\ast(f) &=  \big[ f(1) \big](e)
\end{align}
for $f \in \cO(G)$, $g,h \in |G|$. $\overline{X}$ denotes
the antipode in $\cU(\fg)$.
\end{proposition}

\begin{proof} It is the same as proposition \ref{analsgrps}. Though
the context is different, once the exponential terminology assumes a
meaning for the algebraic category, the calculations are the same.
\end{proof}

Next proposition shows a very natural fact, 
namely that given a SHCP $(G_0,\cO_G)$ the sheaf $\cO_G$ is
the structural sheaf associated with
the superalgebra of its global sections $\cO(G)$, so that the morphisms
$\mu^*$, $i^*$, $e^*$ are actually defined as the appropriate
sheaf morphisms, corresponding to $\mu$, $i$, $e$, multiplication,
inverse and unit in the algebraic supergroup $G=\uspec \cO(G)$.
corresponding to the SHCP
$(G_0,\fg)$.

\begin{proposition}
Let $(G_0,\fg)$ be a SHCP, with $G_0$ an affine group scheme and 
let $\cO_G$ as in \ref{def:SHCP}. Then 
$G:=(G_0, \cO_G)$ is a supergroup scheme.
%The sheaf $\cO_G$ is the structural sheaf associated with $cO(G)$.
\end{proposition} 

\begin{proof}
In proposition \ref{algHopfstructure} we have seen that
$\cO(G):=\Hom_{\cU(\fg_0)}(\cU(\fg),\cO_{G_0}(G_0))$ has an Hopf superalgebra 
structure, moreover by \ref{gamma} it is globally split. Hence
we only need to prove that $G=\uspec \cO(G)$. Clearly the topological
spaces underlying the superspaces
$G=(G_0, \cO_G)$ and  $\uspec \cO(G)$ are homeomorphic. We only need 
to show that
$\cO_{\cO(G)} \cong \cO_G$, where $\cO_{\cO(G)}$ denotes the structural
sheaf associated with the superring $\cO(G)$. 
We set up a morphism:
$$
\begin{array}{cccc}
\phi&\cO_G(U) & \lra &  \cO_{\cO(G)}(U) \\
&s:\cU(\fg) \lra \cO_{G_0}(U) & \mapsto & \phi(s):U \lra 
\coprod_{x \in U} {\cO(G)_x} \\ \\
%\phi(s)(u)=
\end{array}
$$
where $\phi(s)$ is defined as follows. Any $s \in \cO_G(U)$ gives
raise naturally to $s_x: \cU(\fg) \lra \cO_{G_0}(U) \lra \cO_{G_0,x}$.
Since as a $\cU(\fg_0)$ module, $ \cU(\fg)$ is finitely generated, say
by $N$ generators, once we fix those generators, $s_x$ is equivalent
to the choice of $N$ elements in   $\cO_{G_0,x}$. Since likewise
${\cO(G)_x}$ is finitely generated by $N$ elements as free $\cO_{G_0,x}$-module
(those $N$ elements corresponds dually to the generators of $ \cU(\fg)$
as $ \cU(\fg_0)$-module), we have that $s_x$ can be viewed as an element
of ${\cO(G)_x}$. So we define:
$$
\phi(s)(x)=s_x, \qquad x \in U.
$$
We leave to the reader the check that $\phi$ is a sheaf
isomorphism. 
\end{proof}

\medskip

\begin{theorem} \label{alg-eqcat}
The category of algebraic SHCP is equivalent to the 
category of affine algebraic
supergroups.
\end{theorem}

\begin{proof} 
We need to establish a one to one correspondence between
the objects and the morphisms.

As for the objects,
if $(G_0,\fg)$ is an algebraic SHCP, we can define an affine
algebraic supergroup defining
the following Hopf superalgebra (see \ref{algHopfstructure}):
$$
\cO(G_0,\fg)=\uHom_{\cU(\fg_0)} (\cU(\fg),\cO(G_0)).
$$
Vice-versa, if we have an algebraic supergroup, we can find
right away the SHCP associated to it. What we need to show is that
these operations are one the inverse of the other, that is:
$$
\cO(G_0,\fg) \cong \cO(G)
$$
where $G_0$ is the algebraic group underlying $G$ and $\fg=\Lie(G)$.
Certainly they are isomorphic as $\cO(G_0)$-modules, since
they have the same reduced part and, by a result of Masuoka 
\cite{masuoka}, they both
can be written as $\cO(G_0) \otimes \wedge$ for some 
exterior algebra $\wedge$, but being their odd dimension
the same, the two exterior algebras are isomorphic.

\medskip

We can set a map:
$$
\begin{array}{ccc}
\eta^*: \cO(G) & \lra & \cO(G_0,\fg) \\
s & \mapsto & \overline{s}:X \mapsto (-1)^{|X|}|D_X(s)|
\end{array}
$$
where $D_X(s)=(1 \otimes X)\mu^*$.
This is a well defined morphism of Hopf
superalgebras 
and $X \mapsto (-1)^{|X|}|D_X(s)|$ is a $\cU(\fg_0)$-morphism. This
is done precisely in the same way as in the proof of \ref{anal-eqcat}.

\medskip

We now want to show that $\eta^*$ is surjective. This will imply that
$\eta^*$ is an isomorphism. In fact the two given supergroups 
$G=\uspec \cO(G)$ and $\overline{G}=\uspec  \cO(G_0,\fg)$  are smooth
superschemes, with the same underlying topological space and
same Lie superalgebra (hence the same superdimension), and $\eta^*$ induces 
an injective morphism $\eta:\overline{G} \lra G$ (see \cite{fg2} sec. 2).

\medskip

For the surjectivity of $\eta^*$, 
we need to show that for each morphism of $\cU(\fg_0)$-modules
$\overline{s}:\cU(\fg) \lra \cO(G_0)$, 
there exists $s \in \cO(G)$, such that $\overline{s}(X)= (-1)^{|X|}|D_X(s)|$.
Since $\cU(\fg) \cong \cU(\fg_0) \otimes \wedge(\fg_1)$ (see theorem 
\ref{gamma}) and $\overline{s}$ is a morphism of $\cU(\fg_0)$-modules,
$\overline{s}$ is determined by $\overline{s}(\gamma(X^I))$ for
$X^I=X_1^{i_1} \dots X_n^{i_n}$, with $X_i$ a basis for $\fg_1$ and
$i_j=0,1$ (again refer to \ref{gamma}). Notice that
$X_i=\gamma(X_i)$.
Since $X_1, \dots, X_n$ are linearly independent, also the
corresponding left invariant vector fields $D_{X_1}, \dots , D_{X_n}$
will be linearly independent at each point. 
Let $D_{\gamma(X)}$ denote the left invariant differential operator
corresponding to $\gamma(X) \in \cU(\fg)$.
Notice that fixing a suitable basis in $\cU(\fg)$, the 
linear morphism $X \mapsto \gamma(X)$ corresponds to an
upper triangular matrix and sends linearly independent vectors to
linearly independent vectors.
Consider the equation 
$(-1)^{|X^I|}|D_{\gamma(X^I)}s| =\overline{s}(X^I)$, for 
$X^I=X_1^{i_1} \dots {X_n}^{i_n}$ a monomial in 
$\wedge(\fg_1)$. 
This is an equation where each $D_{X_i}$ appearing
in the expression for $D_{\gamma(X^I)}$ can be expressed as
$$
D_{X_i}=\sum a_i \partial_{x_{ij}}, \qquad p(a_i) \neq p(x_{ij})
$$
where 
$x_{ij}$ are global coordinates on
$\rGL_{m|n} \supset G$ (regardless of their parity). 

Since $D_{X_1}^{i_1} \dots D_{X_n}^{i_n}$ are linearly independent
by the PBW theorem (see also proposition \ref{leftinv-operators}),  
also $D_{\ga(X)}$ will be linearly independent and
$(-1)^{|X|}|D_{\gamma(X^I)}|=\overline{s}(X^I)$ will
yield a solution %for all indices $i_1,j_1, \dots i_r,j_r$:
$$
\partial_{x_{i_1j_1}} \dots \partial_{x_{i_rj_r}} s=a_{i_1j_1 \dots i_rj_r}
$$
for all ${i_1j_1} \dots {i_rj_r}$ so that
$$
s=\sum a_{i_1j_1 \dots i_rj_r}x_{i_1j_i} \dots x_{i_rj_r}.
$$
We leave to the reader the correspondence between morphisms.
\end{proof}

\begin{example}
We want to verify explicitly the surjectivity of $\eta^*$ in
the case of $\rGL(1|1)$ and make few remarks on how to extend the
calculation to the case of $G=\rGL(m|n)$. Let $\cO(\rGL(1|1))=
k[a_{11},a_{22}, \al_{12}, \al_{21}][a_{11}^{-1}, a_{22}^{-1}]$.
Let $D_{12}$ and $D_{21}$ denote the left invariant vector fields
corresponding to the generators $\partial_{\al_{12}}$, 
$\partial_{\al_{21}}$ of $\Lie(G)_1$:
$$
\begin{array}{rl}
D_{12}&=\left( 1 \otimes \partial_{\al_{12}}\right) \mu^*=
a_{11} \partial_{\al_{12}}+\al_{21} \partial_{a_{22}} \\ \\
D_{21}&=\left( 1 \otimes \partial_{\al_{21}}\right) \mu^*=
\al_{12} \partial_{a_{11}}+a_{22} \partial_{\al_{21}} \\ \\
\end{array}
$$
%D_{12}D_{21}&=a_{11}\partial_{a_{11}}+\al_{21}\partial_{\al_{21}}+a_{11}a_{22} 
%\partial_{\al_{12}} \partial_{\al_{21}} \\ \\
$$
\begin{array}{rl}
\gamma(D_{12}D_{21})&=1/2(D_{12}D_{21}-D_{21}D_{12})=1/2
(a_{11}\partial_{a_{11}}-a_{22}\partial_{a_{22}})+
\\ \\ 
&+a_{11}a_{22} 
\partial_{\al_{12}} \partial_{\al_{21}}+ \hbox{terms with coefficients
in} \, J_{\cO(\rGL(1|1))}
\end{array}
$$
where $J_{\cO(\rGL(1|1))}$ denotes as usual the ideal generated by
the odd elements.
Notice that the terms with coefficients
in $J_{\cO(\rGL(1|1))}$ do not contribute
in the expression $|D_{\ga(D_{12}D_{21})}s|$. For the same reason, notice
that the term $a_{11}\partial_{a_{11}}-a_{22}\partial_{a_{22}}$ will give
a contribute only if applied to $s^0$, and consequently can be
considered not as unknown, but as a known term.% in our problem
%to determine  $\partial_{\al_{i_1j_1}} \dots \partial_{\al_{i_rj_r}} s$
This is important
in case one wants to generalize this procedure to $\rGL(m|n)$; in fact
only the terms containing only odd
derivations will produce new quantities to be determined.

\medskip
  
Given $\overline{s}:\cU(\fg) \lra \cO(G_0)$ we want to determine
$s \in \cO(G)$, with $\eta^*(s)=\overline{s}$. Since $\Lie(\rGL(1|1)_1=
\langle \partial_{\al_{12}},\partial_{\al_{21}} \rangle$,
$\overline{s}$ is
determined once we know its image on $\wedge \Lie(\rGL(1|1)_1$ that is
$$
s^0=\overline{s}(1), \quad
s^{12}=\overline{s}( \partial_{\al_{12}}), \quad 
s^{21}=\overline{s}(\partial_{\al_{21}}),  \quad 
s^{12,21}=\overline{s}(\gamma(\partial_{\al_{12}}\partial_{\al_{21}})).
$$
Consequently the $s$ we want to determine must satisfy the equations:
$$
\begin{array}{rl}
s^0 &=|1s| \\ \\
s^{12}&=-|a_{11} \partial_{\al_{12}}s+\al_{21} \partial_{a_{22}}s| \\ \\
s^{21}&=-|\al_{12} \partial_{a_{11}}s+a_{22} \partial_{\al_{21}}s| \\ \\
s^{12,21}&=|1/2(a_{11}\partial_{a_{11}}s-a_{22}\partial_{a_{22}}s)+a_{11}a_{22} 
\partial_{\al_{12}} \partial_{\al_{21}}s|
\end{array}
$$
A simple calculation gives us:
$$
\begin{array}{rl}
s &= s^0+ \frac{\al_{12}s^{12}}{a_{11}}- 
 \frac{\al_{21} s^{21}}{a_{22}}+ \\ \\
&+\left[s^{12,21}-\frac{1}{2} 
\left(a_{11}\partial_{a_{11}}s^0-a_{22}\partial_{a_{22}}s^0 \right)
\right] \frac{\al_{12}\al_{21}}{a_{11}a_{22}}.
\end{array}
$$

There is no conceptual obstacle to extend this calculation
to the case of $G=\rGL(m|n)$. If $\cO(G)=k[a_{ij}, \al_{kl}][d_1^{-1},d_2^{-1}]$
where $d_1=det(a_{ij})_{\{1 \leq i,j \leq m\}}$ and 
$d_2=det(a_{ij})_{\{m+1 \leq i,j \leq m+n\}}$, we have that the
left invariant vector fields are given by:
$$
X_{ij}=\left( 1 \otimes \partial_{x_{ij}}\right) \mu^*=
\sum_k x_{ki} \partial_{x_{kj}}
$$
where $x_{ij}$ denote the coordinates on $\rGL(m|n)$ regardless
of their parity.
We can then repeat the calculation we did above. Notice that
any even derivation appearing in the expression $|D_{\gamma(X)}s|$
will affect only $s^0=|1s|$ since we are taking the reduction
modulo the ideal of the odd nilpotents. 
%Consequently such derivations of $s^0$ are known.  
\end{example}

In the following remark we clarify the relation between the 
Hopf superalgebra $\cO(G)=\Hom(\cU(\fg), \cO(G_0))$ associated to
the SHCP $(G_0, \fg)$ and the
distribution superalgebra $D(G)$ of the supergroup $G$ (also naturally
associated to the same SHCP).
 
\begin{remark}  \label{duality}
For an affine supergroup $G$, 
the superalgebra of distributions $D(G)$ has a natural Hopf superalgebra
structure as we detail in \ref{distr-hopf}. Such structure is
inherited by $k|G| \otimes \cU(\fg)$ through the linear isomorphism
with $D(G)$ detailed in \ref{iso}. The superalgebra of global sections of $G$,
$\cO(G)=\Hom(\cU(\fg), \cO(G_0))$ can then be naturally viewed as a subspace of
$D(G)^* \cong (k|G| \otimes \cU(\fg))^*$, 
since elements in $\cO(G)$ arise as suitable morphisms 
$|G| \times \cU(\fg) \lra k$. 
%via the identification $D(G) \cong k|G| \otimes \cU(\fg)$. 
One can then immediately verify
that the Hopf superalgebra structure on $\cO(G) \subset D(G)^*$ 
is precisely obtained by duality, 
from the Hopf superalgebra on $D(G)$ suitably restricting the
comultiplication, counit and antipode morphisms.
\end{remark}

\section{Action of supergroups and SHCP's} 
\label{actions}

In this section
we want to relate the action of an analytic of algebraic
supergroup $G$ on a supermanifold or superscheme $M$, with the action of 
the corresponding SHCP $(G_0, \fg)$ on $M$. We first introduce a 
(well known) definition.

\begin{definition} \label{action-def}
A morphism 
\begin{equation*}
    a \colon G\times M \lra M
\end{equation*}
is called an \emph{action} of $G$ on $M$ if it satisfies
\begin{subequations}
\begin{gather}
    \label{eq:action_mult}
    a \circ ( \mu \times \id_M ) = a \circ ( \id_G \times  a ) \\
    a \circ \langle \hat{e} , \id_M \rangle = \id_M
\end{gather}
\end{subequations}
\end{definition}
\noindent
In the functor of points notation, this is the same as:
\\
1. $1 \cdot x=x$, $\forall x \in M(T)$, $1$ the unit in $G(T)$, \\
2. $(g_1g_2) \cdot x=g_1 \cdot (g_2 \cdot x)$, $\forall x \in M(T)$,
$\forall g_1, g_2 \in G(T)$.
\\
where $T$ is a supermanifold (resp. a superscheme) and
$M(T)=\Hom(T, M)$ are the $T$-points of $M$.

If an action $a$ of $G$ on $M$ is given, then we say that $G$ \textit{acts}
on $M$.%, or that $M$ is a \emph{$G$-supermanifold}.

\medskip

If $(G_0, \fg)$ is an analytic SHCP, we can define what it means for
 $(G_0, \fg)$ to act on a supermanifold.

\begin{definition} \label{actionanalshcp}
We say that the SHCP $(G_0,\fg)$ acts on a supermanifold $M$, if there are
\begin{enumerate}
    \item an action
    \begin{equation} \label{eq:interesting}
        \ua \colon G_0 \times M \lra M
    \end{equation}
    $\ua \coloneqq a \circ (j_{|G|\lra G} \times
    \id_M)$ of the reduced Lie group $G_0$ on the supermanifold
    $M$;
    \item a representation
        \begin{equation} \label{eq:infinitesimalaction}
        \begin{aligned}
            \rho_a \colon \fg &\lra \VEC(M)^\op \\
            X &\MAPSTO \left( X \otimes \id_{\cO(M)} \right) a^*
        \end{aligned}
    \end{equation}
    of the super Lie algebra $\fg$ of $G$ on the opposite of the Lie
    superalgebra of vector fields over $M$.
\end{enumerate}
and the two morphisms satisfy the following compatibility relations
\begin{subequations} \label{eqs:compatibilityforactions}
\begin{align}
    \restr{\rho_a}{\fg_0}(X) &= \left( X \otimes \id_{\cO(M)} \right) \ua^* &
    &\forall X \in \fg_0 \\
    \rho_a(g.Y) &= \big(\ua^{g^{-1}}\big)^* \rho_a(Y) {(\ua^g)}^* &
    &\forall g \in |G|, \, Y \in \fg
\end{align}
\end{subequations}
where  $a^g \colon M \to M$, $a^g  := a \circ 
\langle \hat{g} , \id_M \rangle$
\end{definition}

The next proposition tells us that actions of a SHCP correspond
bijectively to actions of the
corresponding analytic supergroup.

\begin{proposition} \label{prop:action_SLG_SHCP}
Let $G$ be an analytic supergroup acting on a supermanifold $M$.
Then there is an action of the SHCP $(G_0,\Lie(G))$ on $M$.
Conversely, given an action of  the SHCP $(G_0,\fg)$ on $M$,
%let $(G_0,\fg)$ be the SHCP associated with $G$, and
%let maps $\ua$ and $\rho$ like in points (1) and (2)
%above satisfying conditions \eqref{eqs:compatibilityforactions} be
%given. 
there is a unique action $a_\rho \colon G\times M \lra M$ of
the analytic supergroup $G$ corresponding
to the given SHCP on $M$ whose reduced and infinitesimal 
actions are the given ones. If $U$ is an open subset of $M$, we have
\begin{equation} \label{eq:reconstructingactions}
    \begin{aligned}
        a_\rho^* \colon \cO_M(U) &\lra \HOM_{\cU(\fg_0)} \big( \cU(\fg), (\cinfty_{G_0} \hotimes \cO_M)(\topored{a}^{-1}(U)) \big) \\
        f &\MAPSTO \Big[ X \MAPSTO (-1)^{\p{X}} \big( \id_{\cinfty(G_0)} \otimes \rho(X) \big) \ua^*(f) \Big]
    \end{aligned}
\end{equation}
\end{proposition}

\begin{proof}%[Proof of prop.\ \ref{prop:action_SLG_SHCP}]
Let us check that $a_\rho^*(f)$ is $\cU(\fg_0)$-linear. For all $X
\in \cU(\fg)$ and $Z \in \fg_0$ we have
\begin{align*}
    a_\rho^*(f)(ZX)
    &= (-1)^{\p{X}} \big( \id \otimes \rho(ZX) \big) \ua^*(f) \\
    &= (-1)^{\p{X}} \big( \id \otimes \rho(X) \big) (\id \otimes Z_e \otimes \id) (\id \otimes \ua^*) \ua^*(f) \\
    &= (-1)^{\p{X}} \big( \id \otimes \rho(X) \big) (\id \otimes Z_e \otimes \id) (\red{\mu}^* \otimes \id) \ua^*(f) \\
    &= \big( \red{D^{}_Z} \otimes \id \big) \big[ a_\rho^*(f)(X) \big]
\end{align*}
We now check that $a_\rho^*$ is a superalgebra morphism.
\begin{align*}
    \big[ a_\rho^*(f_1) \cdot a_\rho^*(f_2) \big](X)
    &= m_{\cinfty(G_0) \hotimes \cO(M)} \big[ a^*(f_1) \otimes a^*(f_2) \big] \Delta(X) \\
    &= (-1)^{\p{X}} m \Big[ \big(\id \otimes \rho(X_{(1)})\big) \ua^*(f_1) \otimes \big(\id \otimes \rho(X_{(2)})\big) \ua^*(f_2) \Big] \\
    &= (-1)^{\p{X}} \big( \id \otimes \rho(X) \big) \big( \ua^*(f_1) \cdot \ua^*(f_2) \big) \\
    &= a_\rho^*(f_1 \cdot f_2)(X)
\end{align*}
where $f_i \in \cO(M)$ and $X_{(1)} \otimes X_{(2)}$ denotes
$\Delta(X)$. Concerning the ``associative'' property, we have that,
for $X,Y \in \cU(\fg)$ and $g,h \in G_0$,
\begin{align*}
    \big[ (\mu^* \otimes \id) a_\rho^*(f) \big](X,Y)(g,h)
    &= \big[ a_\rho^*(f) \big] (h^{-1}.X Y)(gh) \\
    &= (-1)^{\p{X} + \p{Y} + \p{X}\p{Y}} \rho(Y) \rho(h^{-1}.X) {(\ua^{gh})}^*(f) \\
    &= (-1)^{\p{X} + \p{Y} + \p{X}\p{Y}} \rho(Y) {(\ua^h)}^* \rho(X) {(\ua^g)}^*(f) \\
    &= \big[ (\id \otimes a_\rho^*) a_\rho^*(f) \big](X,Y)(g,h)
\end{align*}
and, finally, $(\ev_e \otimes \id)a_\rho^*(f) = \rho(1) = f$.

{\em Uniqueness} can be proved as follows.
Let $a$ be an  action of $G$ on $M$ and let $(\ua,\rho_a)$ be as in
prop.\ \ref{prop:action_SLG_SHCP}. If $f \in \cO_M(U)$, then
\begin{equation*}
    a^*(f) \in( \HOM_{\cU(\fg_0)} \big( \cU(\fg), \cinfty_{G_0} \big) 
\hotimes \cO_M)(\topored{a}^{-1}(U)) \cong
\end{equation*}
\begin{equation*}
    \cong \HOM_{\cU(\fg_0)} \big( \cU(\fg), (\cinfty_{G_0} \hotimes \cO_M)(\topored{a}^{-1}(U)) \big)
\end{equation*}
hence, using eq.\ \eqref{eq:action_mult} and the fact that $\rho_a$
is an antihomomorphism, for all $X \in \cU(\fg)$
\begin{align*}
    a^*(f)(X)
    &= (-1)^{\p{X}} \big[ (D^{}_X \otimes \id) a^*(\phi) \big](1) \\
    &= (-1)^{\p{X}} \big( \id \otimes \rho_a(X) \big) \big( a^*(f)(1) \big) \\
    &= (-1)^{\p{X}} \big( \id \otimes \rho_a(X) \big) \ua^*(f)
\end{align*}
\end{proof}

Let us now assume $G$ is an affine algebraic supergroup over
a field of characteristic zero and
$(G_0, \fg)$ is the corresponding SHCP and furthermore assume
they are acting on a supervariety
$M$, the definition \ref{actionanalshcp} 
being the same, taking the morphisms
in the appropriate category.

We state the analogue of the proposition  \ref{prop:action_SLG_SHCP}
in the algebraic setting, its proof being essentially the same.

\begin{proposition} \label{prop:action_SHCP}
Let $G$ be an algebraic supergroup acting on a supervariety $M$
(not necessarily affine).
Then there is an action of the SHCP $(G_0,\Lie(G))$ on $M$.
Conversely, given an algebraic action of  the algebraic SHCP $(G_0,\fg)$ on $M$,
%let $(G_0,\fg)$ be the SHCP associated with $G$, and
%let maps $\ua$ and $\rho$ like in points (1) and (2)
%above satisfying conditions \eqref{eqs:compatibilityforactions} be
%given. 
there is a unique action $a_\rho \colon G\times M \lra M$ of
the algebraic supergroup $G$ corresponding
to the given SHCP on $M$ whose reduced and infinitesimal 
actions are the given ones. If $U$ is an open subset of $M$, we have
\begin{equation} \label{eq:reconstructingactions}
    \begin{aligned}
        a_\rho^* \colon \cO_M(U) &\lra \HOM_{\cU(\fg_0)} 
\big( \cU(\fg), (\cO_{G_0} \otimes \cO_M)(\topored{a}^{-1}(U)) \big) \\
        f &\MAPSTO \Big[ X \MAPSTO (-1)^{\p{X}} \big( \id_{\cO(G_0)} 
\otimes \rho(X) \big) \ua^*(f) \Big]
    \end{aligned}
\end{equation}
\end{proposition}

\end{document}